\documentclass[12pt,letterpaper, final]{amsart}
\usepackage{amssymb, amsmath, amsthm}
\usepackage{mathrsfs}  
\usepackage{graphicx}
\usepackage[colorlinks=true, citecolor=blue, linkcolor=red, urlcolor=blue]{hyperref}
\usepackage[text={6in,8in},centering]{geometry}
 \usepackage{xcolor}

\usepackage{hyperref}
\usepackage{mathrsfs}  

\usepackage{bm}

\newcommand{\Id}{\operatorname{Id}}

\usepackage{ifdraft}
\ifoptionfinal{
\usepackage[disable]{todonotes}
}{
\usepackage[norefs, nocites]{refcheck}

\usepackage[notref, notcite]{showkeys}
\usepackage[bordercolor=white, color=white]{todonotes}
}

\makeatletter\providecommand\@dotsep{5}\def\listtodoname{List of Todos}\def\listoftodos{\hypersetup{linkcolor=black}\@starttoc{tdo}\listtodoname\hypersetup{linkcolor=blue}}\makeatother

\usepackage{hyperref}
\usepackage{etoolbox}

\newtheorem{theorem}{Theorem}[section]
\newtheorem{lemma}[theorem]{Lemma}

\newtheorem{proposition}[theorem]{Proposition}
\theoremstyle{definition}

\theoremstyle{remark}
\newtheorem{remark}[theorem]{Remark}

\numberwithin{equation}{section}

\def\C{\mathbb C}
\def\R{\mathbb R}

\renewcommand{\leq}{\leqslant}
\renewcommand{\geq}{\geqslant}

\newcommand*\xbar[1]{%
   \hbox{%
     \vbox{%
       \hrule height 0.5pt 
       \kern0.5ex
       \hbox{%
         \ensuremath{#1}%
       }%
     }%
   }%
} 

\title[Fractional Calder\'on problem]{Fractional Calder\'on problem with a potential}
\author[Feizmohammadi]{Ali Feizmohammadi}
\address{A. Feizmohammadi, Department of Mathematics, University of Toronto, Road Deerfield Hall, 3008K Mississauga, ON L5L 1C6}
\email{ali.feizmohammadi@utoronto.ca}

\author[Krupchyk]{Katya Krupchyk}
\address
        {K. Krupchyk, Department of Mathematics\\
University of California, Irvine\\
CA 92697-3875, USA }

\email{katya.krupchyk@uci.edu}

\author[Uhlmann]{Gunther Uhlmann}
\address
       {G. Uhlmann, Department of Mathematics\\
       University of Washington\\
       Seattle, WA  98195-4350\\
       USA}
\email{gunther@math.washington.edu}

\begin{document}


\begin{abstract}
We prove that the knowledge of local Cauchy data for a fractional Schr\"odinger
equation on a smooth closed and connected Riemannian
manifold of dimension at least two determines the manifold, its
metric, and a smooth complex-valued potential, up to an isometry
fixing the observation set. The observation set may be any
nonempty open subset and no further geometric assumptions are required. The proof uses functional calculus identities to recover the
local action of the resolvent of the Laplace--Beltrami operator.
\end{abstract}

\maketitle

\section{Introduction}

For an equation involving the fractional Laplacian, local
measurements retain information about the global geometry on
which the operator is defined. When the equation also contains
an unknown potential, the inverse problem requires the
simultaneous determination of a nonlocal principal part and
a local perturbation. In this paper, we solve this problem on
smooth closed connected Riemannian manifolds in dimensions two
and higher, without further geometric assumptions. The
measurements may be confined to any nonempty open set, and
the potential need only be known on that set.

The problem is a nonlocal analogue of the geometric anisotropic Calder\'on
problem, in which one seeks to determine a Riemannian metric
from the boundary Cauchy data of harmonic functions; see
\cite{Calderon_1980,Lee_Uhlmann_1989,Lassas_Uhlmann_2001}.
The latter problem remains open for general smooth metrics in
dimensions three and higher. Its known resolutions include
the real-analytic setting and, for the determination of a
conformal factor, suitable transversally anisotropic geometries
\cite{Lee_Uhlmann_1989,Lassas_Uhlmann_2001,DKSU_2009,DKLS_2016}.
For fractional equations, the connection between local
measurements and global spectral information leads to a
different approach to geometric recovery, as we discuss below.

To state our result, let us first introduce the Cauchy data. Let $(M,g)$ be a smooth closed connected Riemannian manifold of dimension $n\geq2$. The Laplace--Beltrami operator $-\Delta_g$ is nonnegative and self-adjoint on $L^2(M,dV_g)$ with domain $H^2(M)$. For $\alpha\in(0,1)$, we define, by the spectral theorem, the fractional Laplacian $(-\Delta_g)^\alpha$ as an unbounded self-adjoint operator on $L^2(M,dV_g)$ with domain $H^{2\alpha}(M)$.

Let $V\in C^\infty(M;\mathbb C)$ and consider the fractional Schr\"odinger operator $(-\Delta_g)^\alpha+V$.
Let $O\subset M$ be a nonempty open set. Following \cite{FKU}, we associate with this operator the Cauchy data set on $O$,
\begin{equation}
\label{eq_data}
\mathcal{C}^O_{M,g,V}
=
\left\{
\left(
u,
(-\Delta_g)^\alpha u\right)\big|_{O}\,:\,
\,u\in C^\infty(M),\,
(-\Delta_g)^\alpha u+Vu=0
\text{ on }M\setminus\overline O
\right\}.
\end{equation}

Our main result is the following.

\begin{theorem}
\label{thm_main}
Let $\alpha\in(0,1)$. For $j=1,2$, let $(M_j,g_j)$ be a smooth closed connected Riemannian manifold of dimension $n\geq2$, and let
$V_j\in C^\infty(M_j;\mathbb C)$. Suppose that $O\subset M_1\cap M_2$ is a common nonempty open set, and assume that
\begin{equation}
\label{eq:knownpotential}
V_1|_O=V_2|_O.
\end{equation}
If
\begin{equation}
\label{eq:dataequal}
\mathcal{C}^O_{M_1,g_1,V_1}
=
\mathcal{C}^O_{M_2,g_2,V_2},
\end{equation}
then there exists a smooth diffeomorphism $\Phi:M_1\to M_2$ such that
\begin{equation}
\label{eq:conclusion}
\Phi|_O=\operatorname{Id},
\qquad
g_1=\Phi^*g_2,
\qquad
V_1=V_2\circ\Phi.
\end{equation}
Conversely, \eqref{eq:conclusion} implies \eqref{eq:dataequal}.
\end{theorem}

The metric on $O$ is recovered as part of the
argument and is not assumed known in advance. The observation
set need not be connected or have a regular boundary.

The corresponding simultaneous recovery problem for the fractional
Schr\"odinger operator with external data in Euclidean space is addressed
in the forthcoming paper~\cite{FKU_Euclidean}.

\begin{remark}
The inverse problem of primary interest is the case
$M_j\setminus\overline O\neq\emptyset$, $j=1,2$, when
the equation in \eqref{eq_data} is imposed on a nonempty
inaccessible open set. We do not require this assumption
in Theorem~\ref{thm_main}, since the proof also covers
the case where $O$ is dense in either manifold.
Indeed, if $M_j\setminus\overline O=\emptyset$, the
equation in \eqref{eq_data} is vacuous, and the corresponding
Cauchy data set is independent of the potential.
The recovery of the manifold and metric still follows
from the argument below, which yields an isometry
$\Phi:M_1\to M_2$ fixing $O$ pointwise.
After this identification, $O$ is dense in both manifolds,
and \eqref{eq:knownpotential} gives
$V_1=V_2\circ\Phi$ everywhere by continuity.
\end{remark}

\begin{remark}
\label{rem_gauge}
The values of the potential $V$ on $O$ cannot be recovered from the Cauchy data set \eqref{eq_data}. Indeed, for every $W\in C^\infty_0(O;\mathbb C)$,
\begin{equation}
\label{eq:invisible}
\mathcal{C}^O_{M,g,V+W}
=
\mathcal{C}^O_{M,g,V}.
\end{equation}
Thus the assumption \eqref{eq:knownpotential} is necessary for a global potential-recovery result from the Cauchy data on $O$.
\end{remark}

\begin{remark}
\label{rem:fulloperator}
One may instead measure the fractional Schr\"odinger
operator itself on $O$. Writing
$P_{g,V}=(-\Delta_g)^\alpha+V$, consider the Cauchy data set
\begin{equation}
\label{eq:full_operator_data}
\widetilde{\mathcal C}^{O}_{M,g,V}
=
\left\{
\bigl(u|_O,(P_{g,V}u)|_O\bigr):
u\in C^\infty(M),\
P_{g,V}u=0 \text{ on }M\setminus\overline O
\right\}.
\end{equation}
For the manifolds, potentials, and observation set of
Theorem~\ref{thm_main}, equality of the data
\eqref{eq:full_operator_data} implies the conclusion of
that theorem without 
the assumption \eqref{eq:knownpotential}.
The converse also holds. Thus these data determine the
entire potential, including its values on $O$, as well
as the manifold and its metric.

Indeed, adding a common $W\in C^\infty_0(O;\mathbb C)$
to the potentials transforms each measured pair $(a,b)$
into $(a,b+Wa)$, and therefore preserves equality of the
two data sets.  Lemma~\ref{lem:removeone} and the arguments
of Lemmas~\ref{lem:commoninverse} and~\ref{lem:localinverse}, applied to these modified pairs,
give a common perturbation for which
$P_j=(-\Delta_{g_j})^\alpha+V_j+W$ are invertible and
their local inverse maps agree. The leading symbols
recover $g_1|_O=g_2|_O$, as in
Lemma~\ref{prop:localmetric}. Symbolic inversion also determines the full symbols of
$P_j$ on $O$ modulo smoothing terms.
Each homogeneous coefficient in the classical symbol of
$(-\Delta_g)^\alpha$ is determined by the local coefficients
of the metric and finitely many of their derivatives;
see \cite{Seeley_1967} and
\cite[Chapter~II, Section~11]{Shubin_2021}.
Since $g_1=g_2$ on $O$, the full symbols of the two
fractional Laplacians agree there modulo smoothing terms.
It follows that multiplication by $V_1-V_2$ is smoothing
on $O$, and hence $V_1|_O=V_2|_O$.
Subtracting the now common term $V_j u|_O$ from the
second component of the original measured pairs in
\eqref{eq:full_operator_data} gives equality of the
Cauchy data in Theorem~\ref{thm_main}, which therefore
applies.
\end{remark}

\subsection{Previous literature}

The fractional Calder\'on problem was introduced in
\cite{GSU_2020}, where an unknown potential in a bounded
Euclidean domain was recovered from exterior measurements
for the fractional Schr\"odinger equation. The proof combines
global unique continuation for the fractional Laplacian with
a Runge approximation property. Recovery of the potential
for a known variable principal part was established in
\cite{GLX_2017}. Subsequent results include uniqueness and
logarithmic stability for rough potentials \cite{RS_2020},
recovery from a single measurement \cite{GRSU_2020}, and
determination of drift terms and more general lower order
differential perturbations \cite{CLR_2020,CMRU_2022}.
Exponential instability for broad classes of potentials and
background metrics was proved in \cite{BCR_2025}.

The Caffarelli--Silvestre extension \cite{CS_2007} realizes
the fractional Laplacian through a local equation in one
additional variable. Its formulation for fractional powers
of elliptic operators \cite{ST_2010} has allowed this
perspective to be used with variable coefficients.
Unique continuation estimates obtained through the
Caffarelli--Silvestre extension have played a central role
in the qualitative and quantitative study of these inverse
problems; see \cite{Ruland_2015,RS_2020}.
For the fractional conductivity equation, a fractional
Liouville transform gives a reduction to potential recovery
\cite{Covi_2020}.

For the recovery of anisotropic leading order coefficients,
one approach reduces the nonlocal inverse problem to its
local elliptic counterpart. In \cite{CGRU_2026}, the
Caffarelli--Silvestre extension was used to show, in
dimensions three and higher, that partial exterior
Dirichlet-to-Neumann data for the fractional Calder\'on
problem determine the full boundary Dirichlet-to-Neumann
map for the corresponding local problem. This reduction
transfers uniqueness results for the local Calder\'on
problem to the nonlocal setting. A quantitative version
of this transfer, yielding stability estimates for
isotropic principal coefficients, was obtained in
\cite{BR_2026}. A related reduction for nonlocal parabolic
operators was established in \cite{LLU_2023}.
For an isotropic diffusion coefficient with an additive
absorption term, approximation results and local and
conditional global uniqueness were obtained in
\cite{LZ_2024}.

A different direction was initiated in \cite{Fei_2024}
and developed in \cite{FGKU}, where the fractional
anisotropic Calder\'on problem on closed manifolds was
resolved in the general smooth setting. The approach of
\cite{FGKU} recovers the local heat kernel for all positive
times and then passes to an inverse problem for the wave
equation. Thus the nonlocal measurements provide access
to evolution and spectral data, for which geometric
determination is available on general smooth manifolds.
The relevant inverse spectral and wave results belong
to the theory of the boundary control method; see
\cite{BK_1992,KKL_2001,HLOS_2018,BKL_2022, KKL_2008}.
This reduction to wave information permits geometric
recovery beyond the settings in which the local elliptic
Calder\'on problem is currently understood.

The subsequent work \cite{CO_2024} treats complete
Riemannian manifolds. A further perspective on the approach
of \cite{FGKU} was given in \cite{Ruland_2025}, where the
Caffarelli--Silvestre extension is used to recover the heat
and wave kernels. The distinction between the two lines
of work is therefore the information recovered from the
nonlocal measurements: the data of a local elliptic inverse
problem in the first case, and spectral or evolution data
in the second. For smooth metrics on $\mathbb R^n$ that
are Euclidean outside a compact set, recovery from partial
exterior measurements was proved in \cite{FGKRSU}, by
both a heat semigroup argument and a proof using the
Caffarelli--Silvestre extension. Related geometric
determination results have been obtained for fractional
connection Laplacians \cite{Chien_2023}, fractional Dirac
operators \cite{QU_2024}, and nonlocal parabolic equations
on closed manifolds \cite{Lin_2026}.

An additive potential outside the fractional power presents
a further difficulty, since $(-\Delta_g)^\alpha+V$ and
$-\Delta_g$ do not commute in general. In particular, this
problem differs from recovery for fractional powers of
a local operator that already contains lower order terms.
In our earlier work \cite{FKU}, we obtained the first
simultaneous recovery result for the manifold, metric,
and potential in this setting, under additional assumptions. An extension to fractional equations
with drift, retaining geometric hypotheses, was obtained
in \cite{Cai_Chen_2025}.

The work \cite{FKU} also established two results of
independent interest: an entanglement principle for
fractional powers of the Laplacian and a uniqueness
theorem for a variant of the Gel'fand inverse spectral
problem. The unique continuation property used in the
present paper is a special case of this entanglement
principle. Further extensions and applications of the
entanglement principle appear in
\cite{FL_2024,Lin_Hyperbolic_2026,LLY_2025}.
The inverse spectral approach was further developed
in \cite{FK_passive}.

\subsection{Outline of the proof}

Our proof follows a different approach from that of our
earlier work \cite{FKU}. We recover the local action of
the resolvent of the Laplace--Beltrami operator for every
positive spectral parameter, providing a direct route
from the fractional Cauchy data to geometric determination.
Let us describe the main steps.

Since the operators $(-\Delta_{g_j})^\alpha+V_j$, $j=1,2$,
need not be invertible, we first remove their kernels by
adding a common real-valued perturbation
$W\in C^\infty_0(O;\mathbb R)$, which leaves both Cauchy
data sets unchanged. The resulting operators
$P_j=(-\Delta_{g_j})^\alpha+Q_j$, where $Q_j=V_j+W$,
are both invertible. This step uses a kernel-dimension
argument, equality of the Cauchy data, and fractional
unique continuation, the latter being the one-function
case of the entanglement principle in \cite{FKU}.
Equality of the Cauchy data then gives equality of the
local source-to-solution maps
$f\mapsto(P_j^{-1}f)|_O$. Their leading symbols determine
the metrics on $O$.

We then use the comparison principle established in
Proposition~\ref{prop:paired} twice. It states that, once
the metrics agree on $O$, two smooth functions on the
respective manifolds with matching fractional Cauchy
pairs on $O$ have identical heat evolutions there for
every positive time. The functions need not vanish on
$O$ or be supported there. The proof combines finite speed of propagation for the
wave equation with the Gaussian representation of the
heat semigroup to obtain exponential smallness of the
difference of the heat evolutions and all its time
derivatives, locally uniformly on $O$, for short times. The heat semigroup representation of
the fractional Laplacian, together with these estimates,
leads to a moment argument, following \cite{FGKU},
which establishes equality of the heat evolutions on
$O$ for all positive times.

First, we apply the comparison principle to $P_j^{-1}f$,
with $f\in C^\infty_0(O)$, and take the Laplace transform
in time to obtain equality on $O$ of
$(-\Delta_{g_j}+z)^{-1}P_j^{-1}f$ for every $z>0$.
Both inverse operators are symmetric with respect to
the complex bilinear pairing $\int_{M_j}ab\,dV_{g_j}$.
Reversing their order by bilinear transposition therefore
gives equality on $O$ of
$P_j^{-1}(-\Delta_{g_j}+z)^{-1}f$.
This argument remains valid for complex-valued potentials.

A second application of the comparison principle,
combined with an ordinary differential equation in time
and the heat semigroup representation of the fractional
Laplacian, eliminates the factor $P_j^{-1}$ and yields
\[
\bigl((-\Delta_{g_1}+z)^{-1}f\bigr)|_O
=
\bigl((-\Delta_{g_2}+z)^{-1}f\bigr)|_O,
\qquad z>0,\quad f\in C^\infty_0(O).
\]
Thus the local resolvent of the Laplacian is recovered
from the original fractional data.

The heat semigroup representation of the resolvent,
together with uniqueness of the Laplace transform,
gives equality of the local heat kernels.
The geometric determination result
\cite[Theorem~3.5 and Remark~3.6]{FGKRSU}
then provides an isometry fixing the entire observation
set $O$; see also \cite{HLOS_2018,SS_2025}.
After identifying the manifolds by the recovered isometry,
fractional unique continuation shows that solutions with
matching Cauchy pairs coincide. A single solution corresponding
to a nonzero source then determines the potential on the
inaccessible region, and the known values of the potential
on $O$ complete the recovery.

The paper is organized as follows. Section~\ref{sec_prelim} collects the functional calculus identities and unique continuation result, and proves the heat comparison principle and the kernel-removal lemma. Section~\ref{sec_proof_main_theorem} carries out the reduction to the local Laplace resolvent and completes the proof of Theorem~\ref{thm_main}.

\section{Preliminaries}
\label{sec_prelim}

In the paper, all functions may be complex-valued, unless stated otherwise. In this section, we collect preliminary facts and prove some results that will be used in the proof of
Theorem~\ref{thm_main}, given in Section~\ref{sec_proof_main_theorem}.

\subsection{Functional calculus formulas for the fractional Laplacian, the resolvent and the heat semigroup}

\label{subsec_functional_calculus}

Let $(M,g)$ be a smooth closed Riemannian manifold of dimension
$n\geq2$, and write $L^2(M)=L^2(M,dV_g)$. Consider the strongly
continuous heat semigroup
$e^{t\Delta_g}:L^2(M)\to L^2(M)$, $t\geq0$, associated with the
nonnegative self-adjoint operator $-\Delta_g$.

By \cite[Theorem~3.5]{St}, its heat kernel
$e^{t\Delta_g}(x,y)\in C^\infty((0,\infty)\times M\times M;\mathbb R)$
is symmetric, $e^{t\Delta_g}(x,y)=e^{t\Delta_g}(y,x)$, and satisfies
$(e^{t\Delta_g}u)(x)=\int_M e^{t\Delta_g}(x,y)u(y)\,dV_g(y)$
for every $u\in L^2(M)$, $t>0$, and $x\in M$. The heat kernel is nonnegative and satisfies
$\int_M e^{t\Delta_g}(x,y)\,dV_g(y)=1$
for every $t>0$ and $x\in M$.
Moreover, the heat semigroup is a contraction on both $L^2(M)$
and $L^\infty(M)$:
\begin{equation}
\label{eq_3_1}
\|e^{t\Delta_g}u\|_{L^2(M)}
\leq \|u\|_{L^2(M)},
\qquad t\geq0,\quad u\in L^2(M),
\end{equation}
and
\begin{equation}
\label{eq:heatLinfty}
\|e^{t\Delta_g}u\|_{L^\infty(M)}
\leq \|u\|_{L^\infty(M)},
\qquad t\geq0,\quad u\in L^\infty(M).
\end{equation}

Let $\alpha\in(0,1)$.  The spectral theorem gives
the representation
\begin{equation}
\label{frac_laplace_heat}
(-\Delta_g)^\alpha u
=
\frac{1}{\Gamma(-\alpha)}
\int_0^\infty (e^{t\Delta_g}-I)u\,\frac{dt}{t^{1+\alpha}},
\end{equation}
for  all $u\in\mathcal D(-\Delta_g)=H^2(M)$. 
The integral in \eqref{frac_laplace_heat} is understood as an $L^2(M)$-valued Bochner integral; see \cite{FKU,Caffarelli_Stinga_2016}.

The resolvent also admits a representation in terms of the heat semigroup. For $z>0$,
\begin{equation}
\label{eq:laplaceresolvent}
(-\Delta_g+z)^{-1}u
=
\int_0^\infty e^{-zt}e^{t\Delta_g}u\,dt,
\end{equation}
for $u\in L^2(M)$. 
Here the integral is again an $L^2(M)$-valued Bochner integral: the integrand is continuous in $L^2(M)$, and \eqref{eq_3_1} gives
$\int_0^\infty\|e^{-zt}e^{t\Delta_g}u\|_{L^2(M)}\,dt
\leq z^{-1}\|u\|_{L^2(M)}$. Formula \eqref{eq:laplaceresolvent} follows from the spectral theorem and the scalar identity $\int_0^\infty e^{-t(z+\lambda)}\,dt=(z+\lambda)^{-1}$, valid for $\lambda\geq0$ and $z>0$.

In the applications below, the representations
\eqref{frac_laplace_heat} and \eqref{eq:laplaceresolvent}
are applied to smooth functions. For the pointwise
form of \eqref{frac_laplace_heat} in this case, see
\cite[Proposition~2.1]{FKU}.
We record that, for $u\in C^\infty(M)$, both integrals
converge in every Sobolev space $H^s(M)$,
$s\in\mathbb R$.

Indeed, fix $s\in\mathbb R$. The spectral norm
$\|(I-\Delta_g)^{s/2}v\|_{L^2(M)}$, initially defined for
$v\in C^\infty(M)$, is equivalent to the standard
$H^s(M)$ norm; see
\cite[Section~4.2.3, p.~103]{Sogge_book_eigenfunctions}.
Here $(I-\Delta_g)^{s/2}$ is defined by the self-adjoint
functional calculus. Since this operator commutes with
$e^{t\Delta_g}$, the spectral theorem and \eqref{eq_3_1}
give
$\|e^{t\Delta_g}v\|_{H^s(M)}
\leq C_s\|v\|_{H^s(M)}$
for all $v\in C^\infty(M)$ and $t\geq0$.
By density, this estimate extends to all $v\in H^s(M)$.
The strong continuity of the heat semigroup on $L^2(M)$,
together with commutation with $(I-\Delta_g)^{s/2}$
and the equivalence of Sobolev norms, also implies that
$t\mapsto e^{t\Delta_g}v$ is continuous from $[0,\infty)$
into $H^s(M)$ for every $v\in H^s(M)$.

For $u\in C^\infty(M)$, the identity
$(e^{t\Delta_g}-I)u=\int_0^t e^{r\Delta_g}\Delta_g u\,dr$
holds in $H^s(M)$ and yields
\[
\|(e^{t\Delta_g}-I)u\|_{H^s(M)}
\leq C_s t\|\Delta_g u\|_{H^s(M)}.
\]
On the other hand,
\[
\|(e^{t\Delta_g}-I)u\|_{H^s(M)}
\leq (C_s+1)\|u\|_{H^s(M)}.
\]
Splitting the integral at $t=1$, we obtain
\[
\int_0^\infty
\|(e^{t\Delta_g}-I)u\|_{H^s(M)}\,\frac{dt}{t^{1+\alpha}}
\leq
\frac{C_s}{1-\alpha}\|\Delta_g u\|_{H^s(M)}
+\frac{C_s+1}{\alpha}\|u\|_{H^s(M)}
<\infty.
\]
Similarly, for $z>0$,
\[
\int_0^\infty
\|e^{-zt}e^{t\Delta_g}u\|_{H^s(M)}\,dt
\leq C_s z^{-1}\|u\|_{H^s(M)}.
\]
Thus both integrals converge as $H^s(M)$-valued Bochner
integrals. Their values agree with those of the original
$L^2(M)$-valued integrals, so
\eqref{frac_laplace_heat} and \eqref{eq:laplaceresolvent}
hold in $H^s(M)$ for every $s\in\mathbb R$.

Choosing $s>n/2$ and using the continuous embedding
$H^s(M)\hookrightarrow C(M)$, we may also regard these
integrals as $C(M)$-valued Bochner integrals.
Since point evaluation is a bounded linear functional on
$C(M)$, it commutes with these integrals.
Consequently, for smooth $u$, both identities hold
pointwise on $M$ and hence on $O$, with evaluation at
each $x\in O$ passed under the integral sign.
We will use this observation repeatedly below.

We will also use the representation of the heat semigroup in terms
of the cosine propagator. For $u\in L^2(M)$,
\begin{equation}
\label{eq:gaussian}
e^{t\Delta_g}u
=
\frac{1}{\sqrt{4\pi t}}
\int_{\mathbb R}
e^{-s^2/(4t)}
\cos(s\sqrt{-\Delta_g})u\,ds,
\qquad t>0.
\end{equation}
Here $\cos(s\sqrt{-\Delta_g})$ is defined by the spectral theorem.
For $u\in C^\infty(M)$, the function
$b(s,x)=(\cos(s\sqrt{-\Delta_g})u)(x)$ belongs to
$C^\infty(\mathbb R\times M)$ and solves
$(\partial_s^2-\Delta_g)b=0$ on $\mathbb R\times M$,
with $b(0,\cdot)=u$ and $\partial_s b(0,\cdot)=0$.
These properties follow by differentiating its spectral
representation.

To justify \eqref{eq:gaussian}, recall the scalar Gaussian identity
\[
e^{-t\lambda}
=
\frac{1}{\sqrt{4\pi t}}
\int_{\mathbb R}
e^{-s^2/(4t)}\cos(s\sqrt{\lambda})\,ds,
\qquad t>0,\quad \lambda\geq0,
\]
which follows from the Fourier transform of the Gaussian;
see also \cite[Section~3, (3.5)]{FGKU}. The cosine propagator is strongly continuous on $L^2(M)$, and there holds
$$\|\cos(s\sqrt{-\Delta_g})u\|_{L^2(M)}\leq\|u\|_{L^2(M)} \quad \forall\, s\in\mathbb R.$$
Consequently, the integral in \eqref{eq:gaussian} converges
as an $L^2(M)$-valued Bochner integral, and applying the
spectral theorem to the scalar identity gives
\eqref{eq:gaussian} in $L^2(M)$.

For smooth $u$, the integral in \eqref{eq:gaussian} also
converges in every Sobolev space.
Indeed, the equivalence of the spectral and standard Sobolev norms,
together with $|\cos(s\sqrt{\lambda})|\leq1$, gives
\begin{equation}
\label{eq:waveSobolev}
\|\cos(s\sqrt{-\Delta_g})v\|_{H^r(M)}
\leq C_r\|v\|_{H^r(M)},
\qquad s\in\mathbb R,\quad v\in H^r(M),
\end{equation}
for every $r\in\mathbb R$, where $C_r$ depends only on
$(M,g)$, $r$, and the choice of Sobolev norm.
The spectral theorem and dominated convergence also show that
$s\mapsto\cos(s\sqrt{-\Delta_g})v$ is continuous in $H^r(M)$.
Consequently,
\[
\frac{1}{\sqrt{4\pi t}}
\int_{\mathbb R}
e^{-s^2/(4t)}
\|\cos(s\sqrt{-\Delta_g})v\|_{H^r(M)}\,ds
\leq C_r\|v\|_{H^r(M)},
\qquad t>0.
\]
Thus, for smooth $u$, \eqref{eq:gaussian} holds as an
$H^r(M)$-valued Bochner integral for every $r\in\mathbb R$.
Choosing $r>n/2$ and using Sobolev embedding, we may also
regard it as a $C(M)$-valued Bochner integral.
Since point evaluation commutes with this integral,
\eqref{eq:gaussian} holds pointwise on $M$, and restriction
to any open subset may be passed under the integral sign.

\subsection{Unique continuation properties for the fractional Laplacian}

We recall the following fractional unique continuation property, which is the case $N=1$ of \cite[Theorem~1.8]{FKU}. 

\begin{proposition}
\label{prop:fracUC}
Let $(M,g)$ be a smooth closed connected Riemannian manifold of dimension $n\geq2$, let $\alpha\in(0,1)$, and let $O\subset M$ be a nonempty open set. If $u\in C^\infty(M;\mathbb C)$ satisfies $u|_O=0$ and $\bigl((-\Delta_g)^\alpha u\bigr)|_O=0$, then
$u=0$ on $M$.
\end{proposition}

\subsection{Heat comparison principle}

We first recall the following finite propagation property for the
wave equation, which will be used in the proof of
Proposition~\ref{prop:paired}.
For a more general statement on complete manifolds, see
\cite[Appendix, Theorem~9]{HLOS_2018}.

\begin{proposition}
\label{prop:finiteprop}
Let $(M,g)$ be a smooth closed Riemannian manifold, let $p\in M$,
and let $T>0$. Suppose that
$w,F\in C^\infty([0,T]\times M;\mathbb C)$ satisfy
$(\partial_s^2-\Delta_g)w=F$ on $(0,T)\times M$.
Set
\[
\mathcal C_{p,T}
=
\{(s,y)\in[0,T)\times M:
d_g(p,y)<T-s\}.
\]
If $w(0,\cdot)=\partial_s w(0,\cdot)=0$ on $B_g(p,T)$
and $F=0$ in $\mathcal C_{p,T}$, then
$w=0$ in $\mathcal C_{p,T}$.
\end{proposition}

The following comparison principle shows that, when the metrics
agree on the observation set, matching fractional Cauchy pairs
imply matching heat evolutions there for all positive times.
For related results and methods of proof, see
\cite[Section~2.1]{FGKU} and
\cite[Proof of Lemma~3.1]{FGKRSU}.

\begin{proposition}
\label{prop:paired}
Let $\alpha\in(0,1)$. For $j=1,2$, let $(M_j,g_j)$ be
smooth closed Riemannian manifolds of dimension $n\geq2$
with a common nonempty open set $O$, and suppose that
$g_1|_O=g_2|_O$.
Let $a_j\in C^\infty(M_j;\mathbb C)$ satisfy
\begin{equation}
\label{eq:pair}
a_1|_O=a_2|_O,\qquad
\bigl((-\Delta_{g_1})^\alpha a_1\bigr)|_O
=
\bigl((-\Delta_{g_2})^\alpha a_2\bigr)|_O.
\end{equation}
Then
\begin{equation}
\label{eq:pairedheat}
(e^{t\Delta_{g_1}}a_1)|_O
=
(e^{t\Delta_{g_2}}a_2)|_O,
\qquad t>0.
\end{equation}
\end{proposition}

\begin{proof}
Write $g$ for the common metric on $O$.
To show \eqref{eq:pairedheat}, we follow the method of
\cite[Section~2.1]{FGKU}, with the required small-time estimates
obtained using 
finite speed of propagation
for the wave equation.
For $m=0,1,2,\ldots$ and $v\in C^\infty(M_j)$, locality gives
$\Delta_g^m(v|_O)=(\Delta_{g_j}^m v)|_O$.
Thus, applying $\Delta_g^m$ to the second identity in
\eqref{eq:pair} and using the commutation of the powers
of each individual Laplacian, we obtain
\begin{equation}
\label{eq_7_1}
\bigl((-\Delta_{g_1})^\alpha\Delta_{g_1}^m a_1\bigr)|_O
=
\bigl((-\Delta_{g_2})^\alpha\Delta_{g_2}^m a_2\bigr)|_O.
\end{equation}

Using \eqref{frac_laplace_heat} in \eqref{eq_7_1},
together with the first identity in \eqref{eq:pair} and locality,
we get
\begin{equation}
\label{eq_7_1_1}
\int_0^\infty
\bigg(
(e^{t\Delta_{g_1}}\Delta_{g_1}^m a_1)(x)
-
(e^{t\Delta_{g_2}}\Delta_{g_2}^m a_2)(x)
\bigg)
\frac{dt}{t^{1+\alpha}}
=0,
\qquad x\in O.
\end{equation}
Here evaluation at $x$ may be passed under the integrals
in \eqref{frac_laplace_heat}, since these integrals converge
in $C(M_j)$ by the Sobolev convergence established above.

For smooth initial data, functional calculus gives
\begin{equation}
\label{eq_7_2_1}
\partial_t^m e^{t\Delta_{g_j}}a_j
=
e^{t\Delta_{g_j}}\Delta_{g_j}^m a_j.
\end{equation}
Setting
\begin{equation}
\label{eq_7_2}
h(t,x)
=
(e^{t\Delta_{g_1}}a_1)(x)
-
(e^{t\Delta_{g_2}}a_2)(x),
\qquad t\geq0,\quad x\in O,
\end{equation}
and using \eqref{eq_7_2_1}, we obtain from
\eqref{eq_7_1_1} that
\begin{equation}
\label{eq_7_3}
\int_0^\infty
\partial_t^m h(t,x)\,\frac{dt}{t^{1+\alpha}}
=0,
\qquad x\in O,
\end{equation}
for $m=0,1,2,\ldots$.
As in \cite[Section~2.1]{FGKU}, we will integrate by parts
in \eqref{eq_7_3}. We claim that there are no contributions
from the endpoints.

To prove this claim, we first establish exponential smallness
of $h$ and all its time derivatives as $t\downarrow0$,
locally uniformly in $O$.
Fix a compact set $K\Subset O$. Choose an open set $U$
such that $K\subset U\Subset O$, and choose
$\chi\in C^\infty_0(O)$ such that $\chi=1$ in a neighborhood
of $\overline U$.
For $m=0,1,2,\ldots$, define
\[
b_{j,m}(s,x)
=
\bigl(\cos(s\sqrt{-\Delta_{g_j}})
       \Delta_{g_j}^m a_j\bigr)(x),
\qquad (s,x)\in\mathbb R\times M_j.
\]
Then $b_{j,m}\in C^\infty(\mathbb R\times M_j)$
and solves the initial value problem
\begin{equation}
\label{eq:waveproblems}
\begin{cases}
(\partial_s^2-\Delta_{g_j})b_{j,m}=0
    &\text{on }\mathbb R\times M_j,\\
b_{j,m}(0,\cdot)=\Delta_{g_j}^m a_j
    &\text{on }M_j,\\
\partial_s b_{j,m}(0,\cdot)=0
    &\text{on }M_j.
\end{cases}
\end{equation}
By locality of the Laplacian, the first identity in
\eqref{eq:pair} implies
$(\Delta_{g_1}^m a_1)|_O=(\Delta_{g_2}^m a_2)|_O$.
Hence $d_m=b_{1,m}-b_{2,m}$, defined on
$\mathbb R\times O$, satisfies
$(\partial_s^2-\Delta_g)d_m=0$ there, with
$d_m(0,\cdot)=\partial_s d_m(0,\cdot)=0$ on $O$.

To apply finite propagation on a single closed manifold, define
\[
\widetilde b_m(s,x)
=
\begin{cases}
\chi(x)d_m(s,x),&x\in O,\\
0,&x\in M_1\setminus O.
\end{cases}
\]
Since $\chi$ is compactly supported in $O$,
$\widetilde b_m\in C^\infty(\mathbb R\times M_1)$.
It satisfies the initial value problem
\[
\begin{cases}
(\partial_s^2-\Delta_{g_1})\widetilde b_m=F_m
    &\text{on }\mathbb R\times M_1,\\
\widetilde b_m(0,\cdot)=0
    &\text{on }M_1,\\
\partial_s\widetilde b_m(0,\cdot)=0
    &\text{on }M_1,
\end{cases}
\]
where $F_m=-[\Delta_g,\chi]d_m$ on $\mathbb R\times O$
and is extended by zero outside $O$.
Since $\chi$ is compactly supported in $O$,
$F_m\in C^\infty(\mathbb R\times M_1)$.
Moreover, $F_m=0$ on $\mathbb R\times U$, since
$\chi=1$ near $\overline U$.

Choose $\delta>0$ such that
$B_{g_1}(x,\delta)\subset U$ for every $x\in K$.
Such a choice is possible by compactness of $K$.
For each $x\in K$, the cone $\mathcal C_{x,\delta}$ from
Proposition~\ref{prop:finiteprop}, formed using the metric $g_1$,
is contained in $[0,\delta)\times U$.
Hence, by Proposition~\ref{prop:finiteprop},
$\widetilde b_m(s,x)=0$ for $x\in K$ and $0\leq s<\delta$.
Since $\chi=1$ on $K$ and the cosine propagators are even in $s$,
we obtain
\begin{equation}
\label{eq:shortwave}
b_{1,m}(s,x)=b_{2,m}(s,x),
\qquad x\in K,\quad |s|<\delta.
\end{equation}
The same $\delta$ works for every $m$: it depends only on
$K$, $U$, and $g_1$.

Fix $r>n/2$. By \eqref{eq:waveSobolev} and Sobolev embedding,
there is a constant $C_r$, depending only on the two manifolds,
$r$, and the Sobolev norms, such that
\begin{equation}
\label{eq_7_6}
\sup_{s\in\mathbb R}\sup_{x\in O}
|b_{1,m}(s,x)-b_{2,m}(s,x)|
\leq B_m,
\qquad
B_m
=
C_r\sum_{j=1}^2
\|\Delta_{g_j}^m a_j\|_{H^r(M_j)}.
\end{equation}

Using \eqref{eq_7_2}, \eqref{eq_7_2_1}, and
\eqref{eq:gaussian}, we get, for $x\in K$ and $t>0$,
\begin{align*}
\partial_t^m h(t,x)
&=
(e^{t\Delta_{g_1}}\Delta_{g_1}^m a_1)(x)
-
(e^{t\Delta_{g_2}}\Delta_{g_2}^m a_2)(x)\\
&=
\frac{1}{\sqrt{4\pi t}}
\int_{\mathbb R}
e^{-s^2/(4t)}
\bigl(b_{1,m}(s,x)-b_{2,m}(s,x)\bigr)\,ds\\
&=
\frac{1}{\sqrt{4\pi t}}
\int_{|s|\geq\delta}
e^{-s^2/(4t)}
\bigl(b_{1,m}(s,x)-b_{2,m}(s,x)\bigr)\,ds.
\end{align*}
In the last equality, we used \eqref{eq:shortwave}.
Hence, using \eqref{eq_7_6} and
$e^{-s^2/(4t)}
\leq e^{-\delta^2/(8t)}e^{-s^2/(8t)}$
when $|s|\geq\delta$, we get
\begin{align*}
\sup_{x\in K}|\partial_t^m h(t,x)|
\leq
\frac{B_m}{\sqrt{4\pi t}}
\int_{|s|\geq\delta}e^{-s^2/(4t)}\,ds
\leq
\frac{B_m e^{-\delta^2/(8t)}}{\sqrt{4\pi t}}
\int_{\mathbb R}e^{-s^2/(8t)}\,ds=
\sqrt{2}\,B_m e^{-\delta^2/(8t)}.
\end{align*}
In particular,
\begin{equation}
\label{eq:smalltime}
\sup_{x\in K}|\partial_t^m h(t,x)|
\leq C_{K,m}e^{-c_K/t},
\qquad 0<t\leq1,\quad m=0,1,2,\ldots,
\end{equation}
where we may take $c_K=\delta^2/8$ and
$C_{K,m}=\sqrt{2}\,B_m$.
The exponent $c_K$ is independent of $m$, $\alpha$, and
$a_1,a_2$. No uniform control of $C_{K,m}$ as $m$ varies
is needed.

Also, \eqref{eq:heatLinfty} gives
\begin{equation}
\label{eq:pairedlarge}
\sup_{x\in O}|\partial_t^m h(t,x)|
\leq
\sum_{j=1}^2
\|\Delta_{g_j}^m a_j\|_{L^\infty(M_j)},
\qquad t>0.
\end{equation}

We now return to \eqref{eq_7_3}.
Fix $x\in O$ and choose $K\Subset O$ containing $x$
in its interior.
By \eqref{eq:smalltime}, all boundary terms at $t=0$
vanish when integrating by parts.
The boundary terms at $t=+\infty$ vanish by
\eqref{eq:pairedlarge} and the negative powers of $t$.
This establishes the claim.

Integrating by parts $m$ times in \eqref{eq_7_3}, we obtain
\begin{equation}
\label{eq:momentsh}
\int_0^\infty
h(t,x)\,\frac{dt}{t^{1+\alpha+m}}
=0,
\qquad m=0,1,2,\ldots.
\end{equation}
To conclude, we follow \cite[Section~2.1]{FGKU}; see also
\cite[Proof of Lemma~3.1]{FGKRSU}.

Making the change of variables $r=1/t$ in
\eqref{eq:momentsh}, we get
\begin{equation}
\label{eq_8_1}
\int_0^\infty r^mF(r)\,dr=0,
\qquad m=0,1,2,\ldots,
\end{equation}
where $F(r)=r^{\alpha-1}h(1/r,x)$ for $r>0$.
By \eqref{eq:smalltime} and \eqref{eq:pairedlarge},
\[
|F(r)|\leq C e^{-cr}r^{\alpha-1},
\qquad r>0,
\]
for some constants $C,c>0$.
Since $\alpha>0$, the extension of $F$ by zero to
$(-\infty,0]$ belongs to $L^1(\mathbb R)$.
Furthermore, by \cite[Theorem~3.3.7]{Lerner_book},
its Fourier transform
\[
\widehat F(\xi)
=
\int_0^\infty F(r)e^{-i\xi r}\,dr
\]
extends holomorphically to $\{\operatorname{Im}\xi<c\}$.
The moment identities \eqref{eq_8_1} imply that $\widehat F$ vanishes
at zero with all derivatives. By holomorphy and uniqueness
of the Fourier transform, it follows that $F(r)=0$ for
$r>0$. Hence $h(t,x)=0$ for every $t>0$.
Since $x\in O$ was arbitrary, \eqref{eq:pairedheat} follows.
\end{proof}

\subsection{Fredholm properties and removal of the kernel}
For any potential $Q\in C^\infty(M;\mathbb C)$, the operator $P=(-\Delta_g)^\alpha+Q:H^{2\alpha}(M)\to L^2(M)$ 
is Fredholm of index zero. Indeed, since $(-\Delta_g)^\alpha$ is nonnegative and self-adjoint, $-1$ belongs to its resolvent
set. Moreover, the graph norm of $(-\Delta_g)^\alpha$ is equivalent to the standard Sobolev norm on $H^{2\alpha}(M)$:
\begin{equation}
\label{eq:graph}
C^{-1}\|u\|_{H^{2\alpha}(M)}
\leq \|u\|_{L^2(M)}+\|(-\Delta_g)^\alpha u\|_{L^2(M)}
\leq C\|u\|_{H^{2\alpha}(M)},
\qquad u\in H^{2\alpha}(M),
\end{equation}
where $C>0$ depends only on $(M,g)$ and $\alpha$; see \cite[Section~4.2.3, p.~103]{Sogge_book_eigenfunctions} and \cite[Theorem~4.4]{St}. For $f\in L^2(M)$, the spectral theorem shows that
$u=\bigl((-\Delta_g)^\alpha+1\bigr)^{-1}f$ satisfies
\[
\|u\|_{L^2(M)}\leq\|f\|_{L^2(M)},
\qquad
\|(-\Delta_g)^\alpha u\|_{L^2(M)}\leq\|f\|_{L^2(M)}.
\]
Together with \eqref{eq:graph}, these bounds show that
$(-\Delta_g)^\alpha+1:H^{2\alpha}(M)\to L^2(M)$
is an isomorphism.

Multiplication by $Q-1$ is compact from $H^{2\alpha}(M)$
to $L^2(M)$, by the compact embedding
$H^{2\alpha}(M)\hookrightarrow L^2(M)$ and the boundedness
of multiplication by $Q-1$ on $L^2(M)$; see
\cite[Theorems~2.3.6 and~2.3.1]{Agranovich_book}.
Thus,
\[
P=\bigl((-\Delta_g)^\alpha+1\bigr)+(Q-1)
\]
is a compact perturbation of an isomorphism and hence is
Fredholm of index zero. Consequently, injectivity of $P$
implies invertibility. For a related Fredholm argument for the realization $(-\Delta_g)^\alpha+V:H^\alpha(M)\to H^{-\alpha}(M)$, see \cite[Proposition~2.2]{FKU}.

The following lemma allows us to remove the kernel of a
fractional Schr\"odinger operator by a real-valued perturbation
supported in the observation set. We will use this result
in the proof of Theorem~\ref{thm_main}.

\begin{lemma}
\label{lem:removeone}
Let $(M,g)$ be a smooth closed connected Riemannian manifold of dimension $n\geq2$, let $\alpha\in(0,1)$, and let $Q\in C^\infty(M;\C)$. For every nonempty open set $O\subset M$, there exists $W\in C^\infty_0(O;\R)$ such that $(-\Delta_g)^\alpha+Q+W:H^{2\alpha}(M)\to L^2(M)$ is invertible.
\end{lemma}

\begin{proof}
All dimensions in this proof are taken over $\mathbb C$.
We argue by minimizing the kernel dimension over real-valued
perturbations supported in $O$. Let
\[
d=\min_{W\in C^\infty_0(O;\mathbb R)}
\dim\ker\bigl((-\Delta_g)^\alpha+Q+W\bigr).
\]
The minimum is attained, since the possible kernel dimensions
form a nonempty subset of the nonnegative integers.
Choose $W_0\in C^\infty_0(O;\mathbb R)$ attaining this minimum,
and set
\[
P_0=(-\Delta_g)^\alpha+Q+W_0.
\]
It suffices to show that $d=0$.

Suppose that $d>0$, and choose $u\in\ker P_0$ with
$\|u\|_{L^2(M)}=1$.
By elliptic regularity, $u$ is smooth.
Proposition~\ref{prop:fracUC}, together with the equation
$P_0u=0$, implies that $u(x_0)\ne0$ for some $x_0\in O$.
Multiplying $u$ by $\overline{u(x_0)}/|u(x_0)|$,
which has modulus one, we may assume that $u(x_0)$
is a positive real number. 
By continuity, $\operatorname{Re}(u^2)>0$ in an open
neighborhood $U\subset O$ of $x_0$.
Choose a nonzero nonnegative function
$\chi\in C^\infty_0(U;\mathbb R)$.
Then
\[
\operatorname{Re}\int_M\chi u^2\,dV_g
=
\int_M\chi\operatorname{Re}(u^2)\,dV_g
>0,
\]
and in particular,
\[
\int_M\chi u^2\,dV_g\ne0.
\]

Let
\[
E=\{w\in\ker P_0:(w,u)_{L^2(M)}=0\},
\]
where $(\cdot,\cdot)_{L^2(M)}$ denotes the usual Hermitian
inner product. Thus $\dim E=d-1$, and
$\ker P_0=\mathbb C u\oplus E$ is an orthogonal decomposition.

Choose a sequence $\varepsilon_\nu>0$ tending to zero.
Since $W_0+\varepsilon_\nu\chi\in C^\infty_0(O; \R)$, the minimality of $d$ gives
\[
\dim\ker(P_0+\varepsilon_\nu\chi)\geq d.
\]

Orthogonality to $E$ imposes at most $d-1$ linear conditions
on this kernel. We may therefore choose
\[
v_\nu\in\ker(P_0+\varepsilon_\nu\chi)\cap E^\perp,
\qquad
\|v_\nu\|_{L^2(M)}=1.
\]

The equation
\[
(-\Delta_g)^\alpha v_\nu
=-(Q+W_0+\varepsilon_\nu\chi)v_\nu
\]
and \eqref{eq:graph} give
\begin{align*}
\|v_\nu\|_{H^{2\alpha}(M)}
&\leq C\bigl(
\|(-\Delta_g)^\alpha v_\nu\|_{L^2(M)}
+\|v_\nu\|_{L^2(M)}\bigr)\\
&\leq C\bigl(
1+\|Q+W_0\|_{L^\infty(M)}
+\varepsilon_\nu\|\chi\|_{L^\infty(M)}\bigr)
\leq C',
\end{align*}
where $C'$ is independent of $\nu$, since
$\varepsilon_\nu\to0$.
Since $H^{2\alpha}(M)$ is a Hilbert space, the bounded
sequence $(v_\nu)$ has a weakly convergent subsequence
in $H^{2\alpha}(M)$.
By the compact embedding
$H^{2\alpha}(M)\hookrightarrow L^2(M)$,
we may pass to a further subsequence such that
\[
v_\nu\rightharpoonup v
\quad\text{in }H^{2\alpha}(M),
\qquad
v_\nu\longrightarrow v
\quad\text{in }L^2(M),
\]
for some $v\in H^{2\alpha}(M)$.
The strong convergence gives $\|v\|_{L^2(M)}=1$.
Moreover, since $P_0:H^{2\alpha}(M)\to L^2(M)$
is bounded, $P_0v_\nu\rightharpoonup P_0v$ in $L^2(M)$.
On the other hand,
\[
\|P_0v_\nu\|_{L^2(M)}
=
\varepsilon_\nu\|\chi v_\nu\|_{L^2(M)}
\leq
\varepsilon_\nu\|\chi\|_{L^\infty(M)}
\longrightarrow0.
\]
Uniqueness of the weak limit therefore gives $P_0v=0$.

Since $v_\nu\perp E$ and $v_\nu\to v$ in $L^2(M)$,
we have
$(v,w)_{L^2(M)}
=\lim_{\nu\to\infty}(v_\nu,w)_{L^2(M)}=0$
for every $w\in E$. Thus $v\perp E$.
Consequently,
\[
v\in\ker P_0\cap E^\perp=\mathbb C u,
\]
so $v=cu$ for some $c\in\mathbb C$ with $|c|=1$.

The operator $(-\Delta_g)^\alpha$ is self-adjoint and
commutes with complex conjugation, and hence is symmetric
for the complex bilinear pairing $\int_M fh\,dV_g$.
Multiplication by $Q+W_0$ is also symmetric for this pairing.
Using $P_0u=0$, we obtain
\[
0
=
\int_M v_\nu P_0u\,dV_g
=
\int_M(P_0v_\nu)u\,dV_g
=
-\varepsilon_\nu\int_M\chi v_\nu u\,dV_g.
\]
Dividing by $\varepsilon_\nu$ and using the strong
$L^2(M)$ convergence, together with $\chi u\in L^2(M)$,
we obtain
\[
0
=
\int_M\chi vu\,dV_g
=
c\int_M\chi u^2\,dV_g,
\]
contradicting $c\ne0$ and the choice of $\chi$.

Thus $d=0$, and $P_0$ is invertible because it is Fredholm
of index zero. Taking $W=W_0$ completes the proof.
\end{proof}

\section{Proof of Theorem \ref{thm_main}}
\label{sec_proof_main_theorem}

The proof of Theorem~\ref{thm_main} is divided into several
steps, each presented in a separate subsection.

Throughout this section, we regard each function in
$C^\infty_0(O)$ as a smooth function on either $M_j$ by
extension by zero, without changing its notation.
This extension is smooth even when $\partial O$ is irregular,
since the support is compactly contained in $O$.

\subsection{Removing the kernels of both fractional Schr\"odinger operators}

We first show that the kernels of both fractional
Schr\"odinger operators can be removed by a common
perturbation supported in $O$.

\begin{lemma}
\label{lem:commoninverse}
Under the assumptions of Theorem~\ref{thm_main}, there exists a common $W\in C^\infty_0(O;\R)$ such that, with
\begin{equation}
\label{eq:modified}
Q_j=V_j+W,\qquad P_j=(-\Delta_{g_j})^\alpha+Q_j,
\end{equation}
the operators $P_j:H^{2\alpha}(M_j)\to L^2(M_j)$ are invertible for $j=1,2$, and
\begin{equation}
\label{eq:dataequal_new}
\mathcal{C}^O_{M_1,g_1,Q_1}
=
\mathcal{C}^O_{M_2,g_2,Q_2}.
\end{equation}
\end{lemma}

\begin{proof}
Apply Lemma~\ref{lem:removeone} on $(M_1,g_1)$ with potential $V_1$ and observation set $O$. This gives $W\in C^\infty_0(O;\R)$ such that $P_1=(-\Delta_{g_1})^\alpha+Q_1:H^{2\alpha}(M_1)\to L^2(M_1)$ is invertible, where $Q_1=V_1+W$. Regard $W$ as a function on $M_2$ by extension by zero, and set $Q_2=V_2+W$. Since adding $W$ does not change either Cauchy data set, \eqref{eq:dataequal_new} follows from \eqref{eq:dataequal} and \eqref{eq:invisible}.

Since $P_2:H^{2\alpha}(M_2)\to L^2(M_2)$ is Fredholm of index zero, it remains to prove its injectivity. Suppose that $v\in H^{2\alpha}(M_2)$ satisfies $P_2v=0$. By Seeley's theorem \cite{Seeley_1967}, $(-\Delta_{g_2})^\alpha$ is an elliptic classical pseudodifferential operator of order $2\alpha$. Since multiplication by $Q_2\in C^\infty(M_2)$ has order zero, $P_2$ is also elliptic of order $2\alpha$. Hence elliptic regularity gives $v\in C^\infty(M_2)$; see \cite[Chapter~7, Section~10]{Taylor_book_vol_II}.

By \eqref{eq:dataequal_new}, there exists $u\in C^\infty(M_1)$ satisfying $P_1u=0$ on $M_1\setminus\overline O$ and
\begin{equation}
\label{eq_2_1}
u|_O=v|_O,\qquad
\bigl((-\Delta_{g_1})^\alpha u\bigr)|_O
=
\bigl((-\Delta_{g_2})^\alpha v\bigr)|_O.
\end{equation}
By \eqref{eq:knownpotential}, we have $Q_1|_O=Q_2|_O$. Thus \eqref{eq_2_1} and the equation $P_2v=0$ imply that $P_1u=0$ on $O$ as well. Since $P_1u$ is smooth and $O\cup(M_1\setminus\overline O)$ is dense in $M_1$, it follows that $P_1u=0$ on all of $M_1$. The invertibility of $P_1$ therefore gives $u=0$.

Consequently, \eqref{eq_2_1} yields $v|_O=0$ and $\bigl((-\Delta_{g_2})^\alpha v\bigr)|_O=0$. Proposition~\ref{prop:fracUC} implies that $v=0$ on $M_2$. Hence $P_2$ is injective and therefore invertible.
\end{proof}

\subsection{From Cauchy data to local source-to-solution maps}

By Lemma~\ref{lem:commoninverse}, we may work with the operators $P_j=(-\Delta_{g_j})^\alpha+Q_j$, where $Q_j=V_j+W$, such that $P_j:H^{2\alpha}(M_j)\to L^2(M_j)$ is invertible for $j=1,2$ and \eqref{eq:dataequal_new} holds. Moreover, \eqref{eq:knownpotential} implies that $Q_1|_O=Q_2|_O$.

We denote the inverses by $P_j^{-1}:L^2(M_j)\to H^{2\alpha}(M_j)$. These inverses preserve smoothness. Indeed, if $f\in C^\infty(M_j)$, then $u=P_j^{-1}f\in H^{2\alpha}(M_j)$ satisfies $P_ju=f$. Since $P_j$ is elliptic of order $2\alpha$, elliptic regularity gives $u\in C^\infty(M_j)$; see \cite[Chapter~7, Section~10]{Taylor_book_vol_II}. We use the same notation for the restrictions $P_j^{-1}:C^\infty(M_j)\to C^\infty(M_j)$.

For sources $f\in C^\infty_0(O)$, extended by zero to $M_j$, the local source-to-solution map is $f\mapsto(P_j^{-1}f)|_O$. The following lemma identifies these maps from the Cauchy data.

\begin{lemma}
\label{lem:localinverse}
Suppose that $P_j:H^{2\alpha}(M_j)\to L^2(M_j)$ is invertible for $j=1,2$, that $Q_1|_O=Q_2|_O$, and that \eqref{eq:dataequal_new} holds. Then, for every $f\in C^\infty_0(O)$,
\begin{align}
(P_1^{-1}f)|_O
&=(P_2^{-1}f)|_O,
\label{eq:localinverse}\\
\bigl((-\Delta_{g_1})^\alpha P_1^{-1}f\bigr)|_O
&=\bigl((-\Delta_{g_2})^\alpha P_2^{-1}f\bigr)|_O.
\label{eq:localfractional}
\end{align}
\end{lemma}

\begin{proof}
Let $f\in C^\infty_0(O)$ and set $u_1=P_1^{-1}f\in C^\infty(M_1)$. Since $P_1u_1=f$ on $M_1$ and $f$ is supported in $O$, we have $P_1u_1=0$ on $M_1\setminus\overline O$. By \eqref{eq:dataequal_new}, there exists $u_2\in C^\infty(M_2)$ satisfying
\begin{equation}
\label{eq_2_2}
P_2u_2=0\quad\text{on }M_2\setminus\overline O
\end{equation}
and
\begin{equation}
\label{eq_2_3}
u_1|_O=u_2|_O,\qquad
\bigl((-\Delta_{g_1})^\alpha u_1\bigr)|_O
=
\bigl((-\Delta_{g_2})^\alpha u_2\bigr)|_O.
\end{equation}
Since $Q_1|_O=Q_2|_O$, these identities imply that $P_2u_2=P_1u_1=f$ on $O$. On $M_2\setminus\overline O$, both $P_2u_2$ and $f$ vanish. Thus the smooth function $P_2u_2-f$ vanishes on the dense set $O\cup(M_2\setminus\overline O)$, and consequently $P_2u_2=f$ on all of $M_2$.

The invertibility of $P_2$ now gives $u_2=P_2^{-1}f$. Substituting this and $u_1=P_1^{-1}f$ into \eqref{eq_2_3} proves \eqref{eq:localinverse} and \eqref{eq:localfractional}.
\end{proof}

\subsection{From local source-to-solution maps to heat and resolvent identities}

Our next goal is to recover the local action of the resolvent of the Laplacian. We begin by showing that the local source-to-solution map determines the metric on the
observation set. For the fractional Laplacian without a potential, it was shown in
\cite[Corollary~1.1]{Cheng_Ruland} that a density-weighted local source-to-solution map determines the metric on the
observation set. We give a direct pseudodifferential proof
in the present setting of fractional Schr\"odinger operators
with potentials.

\begin{lemma}
\label{prop:localmetric}
If \eqref{eq:localinverse} holds for every
$f\in C^\infty_0(O)$, then $g_1|_O=g_2|_O$.
\end{lemma}

\begin{proof}
We write $\Psi^m$ for the pseudodifferential class of
type $(1,0)$ and $\Psi_{\mathrm{cl}}^m$ for its classical
subclass. By Seeley's theorem \cite{Seeley_1967},
$(-\Delta_{g_j})^\alpha$ is a classical elliptic
pseudodifferential operator of order $2\alpha$, with
homogeneous principal symbol
\[
|\xi|_{g_j}^{2\alpha}
=
\left(
\sum_{k,l=1}^n g_j^{kl}(x)\xi_k\xi_l
\right)^\alpha,
\qquad
(x,\xi)\in T^*M_j,\quad \xi\ne0,
\]
where $(g_j^{kl})$ is the inverse of the matrix of the metric
in local coordinates. Choose a classical parametrix
$B_j\in\Psi_{\mathrm{cl}}^{-2\alpha}(M_j)$, with homogeneous
principal symbol $|\xi|_{g_j}^{-2\alpha}$, such that
\begin{equation}
\label{eq_9_1}
(-\Delta_{g_j})^\alpha B_j=I+R_j,
\qquad R_j\in\Psi^{-\infty}(M_j);
\end{equation}
see \cite[Chapter~I, Sections~5.4--5.5]{Shubin_2021}.

Since $P_j$ is elliptic of order $2\alpha$ and invertible,
the inverse operator theorem
\cite[Chapter~I, Theorem~8.2]{Shubin_2021} gives
$P_j^{-1}\in\Psi^{-2\alpha}(M_j)$.
Using \eqref{eq_9_1} and the fact that multiplication by
$Q_j$ has order zero, we obtain
\begin{equation}
\label{eq_9_2}
P_j^{-1}-B_j
=
-P_j^{-1}Q_jB_j-P_j^{-1}R_j
\in\Psi^{-4\alpha}(M_j),
\qquad j=1,2.
\end{equation}

Fix $x_0\in O$, choose a coordinate neighborhood
$U\Subset O$ of $x_0$, and take
$\chi\in C^\infty_0(U;\mathbb R)$ with $\chi=1$ near $x_0$.
By \eqref{eq:localinverse},
\begin{equation}
\label{eq_9_3}
\chi P_1^{-1}\chi=\chi P_2^{-1}\chi
\end{equation}
as operators on scalar functions in these common coordinates.
Combining \eqref{eq_9_2} and \eqref{eq_9_3}, we obtain
\[
\chi B_1\chi-\chi B_2\chi\in\Psi^{-4\alpha}(U).
\]
On the other hand, the classical composition calculus
\cite[Chapter~I, Theorem~3.4 and Proposition~3.7]{Shubin_2021}
shows that this difference belongs to
$\Psi_{\mathrm{cl}}^{-2\alpha}(U)$, with homogeneous
principal symbol
\[
\chi(x)^2
\bigl(
|\xi|_{g_1}^{-2\alpha}
-
|\xi|_{g_2}^{-2\alpha}
\bigr).
\]
Since $-4\alpha<-2\alpha$, this homogeneous symbol must
vanish. Evaluating at $x_0$, where $\chi=1$, gives
\[
|\xi|_{g_1(x_0)}^{-2\alpha}
=
|\xi|_{g_2(x_0)}^{-2\alpha},
\qquad
\xi\in T_{x_0}^*O\setminus\{0\}.
\]
As $\alpha>0$, the corresponding positive quadratic forms
are equal, and hence $g_1(x_0)=g_2(x_0)$.
Since $x_0\in O$ was arbitrary, we conclude that
$g_1|_O=g_2|_O$.
\end{proof}

We now have $g_1=g_2$ and $Q_1=Q_2$ on $O$, together with
\eqref{eq:localinverse}--\eqref{eq:localfractional}.
The following lemma converts these identities into equalities for the heat
semigroup and the resolvent of the Laplacian applied to
the solutions $P_j^{-1}f$.

\begin{lemma}
\label{lem:mixed}
Assume that $g_1|_O=g_2|_O$ and that
\eqref{eq:localinverse}--\eqref{eq:localfractional}
hold for every $f\in C^\infty_0(O)$.
Then, for every $f\in C^\infty_0(O)$, $t>0$, and $z>0$,
\begin{align}
(e^{t\Delta_{g_1}}P_1^{-1}f)|_O
&=(e^{t\Delta_{g_2}}P_2^{-1}f)|_O,
\label{eq:mixedheat}\\
\bigl((-\Delta_{g_1}+z)^{-1}P_1^{-1}f\bigr)|_O
&=\bigl((-\Delta_{g_2}+z)^{-1}P_2^{-1}f\bigr)|_O.
\label{eq:mixedresolvent}
\end{align}
\end{lemma}

\begin{proof}
Fix $f\in C^\infty_0(O)$ and set $u_j=P_j^{-1}f\in C^\infty(M_j)$, $j=1,2$. By~\eqref{eq:localinverse}--\eqref{eq:localfractional}, the functions $u_1$ and $u_2$
have matching fractional Cauchy pairs on $O$. Since $g_1=g_2$ on $O$, Proposition~\ref{prop:paired} gives \eqref{eq:mixedheat}.

To prove \eqref{eq:mixedresolvent}, fix $z>0$.
Using \eqref{eq:laplaceresolvent} pointwise, as justified
in Subsection~\ref{subsec_functional_calculus}, we obtain,
for every $x\in O$,
\begin{align*}
\bigl((-\Delta_{g_1}+z)^{-1}u_1\bigr)(x)
&=\int_0^\infty e^{-zt}(e^{t\Delta_{g_1}}u_1)(x)\,dt=\int_0^\infty e^{-zt}(e^{t\Delta_{g_2}}u_2)(x)\,dt\\
&=\bigl((-\Delta_{g_2}+z)^{-1}u_2\bigr)(x),
\end{align*}
where the middle equality follows from \eqref{eq:mixedheat}.
This proves \eqref{eq:mixedresolvent}.
\end{proof}

By Lemma~\ref{lem:mixed}, we know the local action of
$(-\Delta_{g_j}+z)^{-1}P_j^{-1}$.
The next lemma recovers the product in the opposite
order by applying bilinear symmetry twice.
This identity will be used in the next step to recover
the local resolvent of the Laplacian.

\begin{lemma}
\label{lem:reversed}
Assume that $g_1|_O=g_2|_O$ and that
\eqref{eq:mixedresolvent} holds for every $z>0$ and
$f\in C^\infty_0(O)$.
Then, for every $z>0$ and $f\in C^\infty_0(O)$,
\begin{equation}
\label{eq:reversed}
\bigl(P_1^{-1}(-\Delta_{g_1}+z)^{-1}f\bigr)|_O
=
\bigl(P_2^{-1}(-\Delta_{g_2}+z)^{-1}f\bigr)|_O.
\end{equation}
\end{lemma}

\begin{proof}
Throughout this proof, we use the complex bilinear pairing
$(a,b)\mapsto\int_{M_j}ab\,dV_{g_j}$, without complex
conjugation.

As noted in the proof of Lemma~\ref{lem:removeone},
$P_j$ is symmetric with respect to this bilinear pairing.
Since $P_j^{-1}$ maps smooth functions to smooth functions,
this symmetry gives, for all $a,b\in C^\infty(M_j)$,
\begin{align}
\int_{M_j}(P_j^{-1}a)b\,dV_{g_j}
&=\int_{M_j}(P_j^{-1}a)P_j(P_j^{-1}b)\,dV_{g_j} =\int_{M_j}P_j(P_j^{-1}a)(P_j^{-1}b)\,dV_{g_j}
\notag\\
&=\int_{M_j}a(P_j^{-1}b)\,dV_{g_j}.
\label{eq:inversesymmetry}
\end{align}
Similarly, integration by parts for $-\Delta_{g_j}+z$
and smoothness of its inverse on smooth functions yield
\begin{equation}
\label{eq:resolventsymmetry}
\int_{M_j}\bigl((-\Delta_{g_j}+z)^{-1}a\bigr)b\,dV_{g_j}
=
\int_{M_j}a\bigl((-\Delta_{g_j}+z)^{-1}b\bigr)\,dV_{g_j},
\end{equation}
for all $a,b\in C^\infty(M_j)$ and $z>0$.

Fix $z>0$ and $f,q\in C^\infty_0(O)$, and write
$g=g_1|_O=g_2|_O$.
Using \eqref{eq:resolventsymmetry} and then
\eqref{eq:inversesymmetry}, we obtain
\begin{align}
\int_O f\,(-\Delta_{g_j}&+z)^{-1}P_j^{-1}q\,dV_g
=\int_{M_j}f\,(-\Delta_{g_j}+z)^{-1}P_j^{-1}q\,dV_{g_j}
\notag\\
&=\int_{M_j}\bigl((-\Delta_{g_j}+z)^{-1}f\bigr)
                 (P_j^{-1}q)\,dV_{g_j}
=\int_{M_j}q\,P_j^{-1}(-\Delta_{g_j}+z)^{-1}f\,dV_{g_j}
\notag\\
&=\int_O q\,P_j^{-1}(-\Delta_{g_j}+z)^{-1}f\,dV_g.
\label{eq:twosymmetries}
\end{align}
The first and last equalities use the supports of $f$ and $q$,
respectively, together with the equality of the volume
measures on $O$.

By \eqref{eq:mixedresolvent}, applied to the source $q$,
the first integral in \eqref{eq:twosymmetries} is the same
for $j=1$ and $j=2$.
Hence, for every $q\in C^\infty_0(O)$,
\[
\int_O q\Bigl(
P_1^{-1}(-\Delta_{g_1}+z)^{-1}f
-
P_2^{-1}(-\Delta_{g_2}+z)^{-1}f
\Bigr)\,dV_g=0.
\]
The expression in parentheses is smooth on $O$.
Since its pairing with every test function vanishes,
it vanishes pointwise on $O$.
This proves \eqref{eq:reversed}.
\end{proof}

We next eliminate the potential and obtain equality of the local
resolvents of the Laplacians.

\begin{lemma}
\label{lem:bare}
Assume that $g_1|_O=g_2|_O$, $Q_1|_O=Q_2|_O$, and that
\eqref{eq:reversed} holds for every $z>0$ and
$f\in C^\infty_0(O)$.
Then, for every $z>0$ and $f\in C^\infty_0(O)$,
\begin{equation}
\label{eq:bare}
\bigl((-\Delta_{g_1}+z)^{-1}f\bigr)|_O
=
\bigl((-\Delta_{g_2}+z)^{-1}f\bigr)|_O.
\end{equation}
\end{lemma}

\begin{proof}
Fix $z>0$ and $f\in C^\infty_0(O)$, and set
\begin{equation}
\label{eq:wv}
w_j=P_j^{-1}(-\Delta_{g_j}+z)^{-1}f,
\qquad
v_j=(-\Delta_{g_j}+z)w_j.
\end{equation}
All these functions are smooth. By \eqref{eq:reversed},
\begin{equation}
\label{eq_traces_w}
w_1|_O=w_2|_O.
\end{equation}
To obtain the desired identity, we need to show that
$(P_1w_1)|_O=(P_2w_2)|_O$.
Since $P_j$ contains the nonlocal operator
$(-\Delta_{g_j})^\alpha$, this does not follow directly
from \eqref{eq_traces_w}.
The purpose of $v_j$ is to obtain matching fractional
Cauchy pairs by a local calculation.

We first show that the fractional Cauchy pairs of $v_1$
and $v_2$ coincide on $O$.
Since $g_1=g_2$ on $O$ and $w_1|_O=w_2|_O$, locality
gives $v_1|_O=v_2|_O$.
Furthermore, the definition of $w_j$ gives
$P_jw_j=(-\Delta_{g_j}+z)^{-1}f$, and hence
$(-\Delta_{g_j}+z)P_jw_j=f$.
Commuting functions of the same Laplacian, we obtain
\begin{align*}
(-\Delta_{g_j})^\alpha v_j=(-\Delta_{g_j}+z)(-\Delta_{g_j})^\alpha w_j&=(-\Delta_{g_j}+z)(P_jw_j-Q_jw_j)\\
&=f-(-\Delta_{g_j}+z)(Q_jw_j).
\end{align*}
Since $-\Delta_{g_j}+z$ is a differential operator and
$g_1=g_2$, $Q_1=Q_2$, and $w_1=w_2$ on $O$,
the right-hand sides agree there. Thus,
\begin{equation}
\label{eq:vpair}
v_1|_O=v_2|_O,\qquad
\bigl((-\Delta_{g_1})^\alpha v_1\bigr)|_O
=
\bigl((-\Delta_{g_2})^\alpha v_2\bigr)|_O.
\end{equation}

Applying Proposition~\ref{prop:paired} to \eqref{eq:vpair},
we obtain
\begin{equation}
\label{eq:vheat}
(e^{t\Delta_{g_1}}v_1)|_O
=
(e^{t\Delta_{g_2}}v_2)|_O,
\qquad t>0.
\end{equation}

We next pass from the heat evolutions of $v_j$ to those
of $w_j$. Fix $x\in O$ and write
$d(t,x)=(e^{t\Delta_{g_1}}w_1)(x)
-(e^{t\Delta_{g_2}}w_2)(x)$ for $t>0$.
The definition of $v_j$ in \eqref{eq:wv} and the heat
equation give
\[
e^{t\Delta_{g_j}}v_j
=
(-\Delta_{g_j}+z)e^{t\Delta_{g_j}}w_j
=
(-\partial_t+z)e^{t\Delta_{g_j}}w_j.
\]
Consequently, \eqref{eq:vheat} gives
$(-\partial_t+z)d(t,x)=0$ for $t>0$, and therefore
$\partial_t(e^{-zt}d(t,x))=0$.
By the strong continuity of the heat semigroup on
$H^s(M_j)$ and Sobolev embedding with $s>n/2$,
we have $(e^{t\Delta_{g_j}}w_j)(x)\to w_j(x)$
as $t\downarrow0$.
Thus $d(\,\cdot\,,x)$ extends continuously to $t=0$,
with $d(0,x)=w_1(x)-w_2(x)=0$.
It follows that $d(t,x)=0$ for every $t>0$.
This proves
\begin{equation}
\label{eq:wheat}
(e^{t\Delta_{g_1}}w_1)|_O
=
(e^{t\Delta_{g_2}}w_2)|_O,
\qquad t>0.
\end{equation}

Finally, we return to the fractional power representation.
Since $w_j$ is smooth, the integral in
\eqref{frac_laplace_heat} may be restricted to $O$,
as justified in Subsection~\ref{subsec_functional_calculus}.
Using \eqref{eq_traces_w} and \eqref{eq:wheat}, we obtain
\[
\bigl((-\Delta_{g_1})^\alpha w_1\bigr)|_O
=
\bigl((-\Delta_{g_2})^\alpha w_2\bigr)|_O.
\]
Since $Q_1w_1=Q_2w_2$ on $O$, adding these common
potential terms gives $(P_1w_1)|_O=(P_2w_2)|_O$.
Together with $P_jw_j=(-\Delta_{g_j}+z)^{-1}f$,
this yields the desired identity \eqref{eq:bare}.
\end{proof}

\subsection{Recovery of the heat kernels, the metric, and the potential}

We first recover the heat kernels on $O\times O$ for all
positive times, before completing the recovery of the metric
and the potential.

\begin{lemma}
\label{lem:heatwave}
Assume that $g_1|_O=g_2|_O$ and that \eqref{eq:bare}
holds for every $z>0$ and $f\in C^\infty_0(O)$.
Then the heat kernels satisfy
\begin{equation}
\label{eq:heatkernel}
e^{t\Delta_{g_1}}(x,y)=e^{t\Delta_{g_2}}(x,y),
\qquad t>0,\quad x,y\in O.
\end{equation}
\end{lemma}

\begin{proof}
Fix $f\in C^\infty_0(O)$ and $x\in O$.
The function
$t\mapsto (e^{t\Delta_{g_1}}f)(x)
-(e^{t\Delta_{g_2}}f)(x)$
is continuous on $[0,\infty)$ by the strong continuity
of the heat semigroup on $H^s(M_j)$ and Sobolev embedding
with $s>n/2$.
By \eqref{eq:heatLinfty}, it is also bounded. 
Using the pointwise resolvent representation
\eqref{eq:laplaceresolvent} and \eqref{eq:bare},
we obtain, for every $z>0$,
\[
\begin{aligned}
&\int_0^\infty e^{-zt}
\bigl((e^{t\Delta_{g_1}}f)(x)
-(e^{t\Delta_{g_2}}f)(x)\bigr)\,dt\\
&\qquad=
\bigl((-\Delta_{g_1}+z)^{-1}f\bigr)(x)
-
\bigl((-\Delta_{g_2}+z)^{-1}f\bigr)(x)
=0.
\end{aligned}
\]
Uniqueness of the Laplace transform for bounded continuous
functions therefore gives
$(e^{t\Delta_{g_1}}f)(x)=(e^{t\Delta_{g_2}}f)(x)$
for every $t\geq0$.

Now fix $t>0$ and $x\in O$, and write
$g=g_1|_O=g_2|_O$.
The heat kernel representation and the support of $f$ give
$\int_O
\bigl(e^{t\Delta_{g_1}}(x,y)-e^{t\Delta_{g_2}}(x,y)\bigr)
f(y)\,dV_g(y)=0$
for every $f\in C^\infty_0(O)$.
Since the difference of the kernels is smooth as a function
of $y\in O$, it vanishes pointwise there.
This proves \eqref{eq:heatkernel}.
\end{proof}

Applying \cite[Theorem~3.5 and Remark~3.6]{FGKRSU}
to \eqref{eq:heatkernel}, we obtain a smooth diffeomorphism
$\Phi:M_1\to M_2$ such that $\Phi|_O=\Id$ and
$g_1=\Phi^*g_2$ on $M_1$.
For the underlying recovery argument, see the proof of
\cite[Theorem~1.5]{FGKU}, which is based on
\cite[Theorem~2]{HLOS_2018}.
For the conclusion $\Phi|_O=\Id$ when $O$ is an arbitrary
nonempty open observation set, see also
\cite[Theorem~1.1]{SS_2025}.

The recovery of the potentials follows the one-solution
argument of \cite[Section~6]{FKU}, which we include
for completeness.
Let $\tilde Q_2=Q_2\circ\Phi$.
By \eqref{eq:dataequal_new} and the invariance of the
fractional Cauchy data under isometries, as shown in the
proof of \cite[Lemma~A.1]{FKU}, we have
\begin{equation}
\label{eq_4_1}
\mathcal{C}^O_{M_1,g_1,Q_1}
=
\mathcal{C}^O_{M_2,g_2,Q_2}
=
\mathcal{C}^O_{M_1,\Phi^*g_2,\tilde Q_2}
=
\mathcal{C}^O_{M_1,g_1,\tilde Q_2}.
\end{equation}

Choose a nonzero function $f\in C^\infty_0(O)$ and let
$u_1=((-\Delta_{g_1})^\alpha+Q_1)^{-1}f
\in C^\infty(M_1)$.
Since $f$ is supported in $O$, we have
$((-\Delta_{g_1})^\alpha+Q_1)u_1=0$
on $M_1\setminus\overline O$.
Thus \eqref{eq_4_1} implies that there exists
$u_2\in C^\infty(M_1)$ satisfying
$((-\Delta_{g_1})^\alpha+\tilde Q_2)u_2=0$
on $M_1\setminus\overline O$, with
$(u_1-u_2)|_O=0$ and
$((-\Delta_{g_1})^\alpha(u_1-u_2))|_O=0$.
Proposition~\ref{prop:fracUC} gives $u_1=u_2$ on $M_1$,
and subtracting their equations yields
$(Q_1-\tilde Q_2)u_1=0$ on $M_1\setminus\overline O$.

By Proposition~\ref{prop:fracUC} and the fact that $f\ne0$,
the set
$\{x\in M_1\setminus\overline O:u_1(x)\ne0\}$
is dense in $M_1\setminus\overline O$.
Thus, using $(Q_1-\tilde Q_2)u_1=0$ there,
the equality $Q_1=\tilde Q_2$ on $O$, and continuity,
we obtain $Q_1=\tilde Q_2$ on $M_1$.
Since $Q_j=V_j+W$ and $W\circ\Phi=W$, we conclude that
$V_1=V_2\circ\Phi$ on $M_1$.

Conversely, \eqref{eq:conclusion} implies
\eqref{eq:dataequal} by the invariance of the fractional
Cauchy data under isometries.
This completes the proof of Theorem~\ref{thm_main}.

\subsection*{Acknowledgements}
The research of K.K. is partially supported by the National Science Foundation (DMS 2408793).  The research of G.U. is partially supported by the National Science Foundation. 
The authors acknowledge the use of AI tools. Any suggestions incorporated into the manuscript were checked and revised by the authors, who take full responsibility for the content of the manuscript.

\end{document}